\documentclass{article}

\usepackage{setspace}
\usepackage{amssymb,amsmath,amsthm}
\usepackage{subfigure}
\usepackage{calc}
\usepackage{verbatim}
\usepackage{enumerate}
\usepackage{cite}
\usepackage{epsfig}
\usepackage{graphicx}
\usepackage{graphics}
\usepackage {tikz}
\usetikzlibrary{automata,arrows,calc,positioning}
\usetikzlibrary{decorations.pathreplacing}
\usepackage{xcolor}\usepackage[hmargin=2cm,vmargin=2.5cm]{geometry}
\usepackage{chngcntr}
\usepackage[]{algorithm2e}
\counterwithin{figure}{section}

\newtheorem{definition}{Definition}
\newtheorem{lemma}[definition]{Lemma}
\newtheorem{proposition}[definition]{Proposition}
\newtheorem{theorem}[definition]{Theorem}
\newtheorem{corollary}[definition]{Corollary}
\newtheorem{observation}[definition]{Observation}

\newcommand{\avd}{\text{\rm{avd}}}
\newcommand{\um}{\text{\rm{unmatched}}}

\makeatletter
\renewcommand{\pod}[1]{\mathchoice
  {\allowbreak \if@display \mkern 18mu\else \mkern 8mu\fi (#1)}
  {\allowbreak \if@display \mkern 18mu\else \mkern 8mu\fi (#1)}
  {\mkern4mu(#1)}
  {\mkern4mu(#1)}
}

\title{On the unimodality of nearly well-dominated trees}

\author{Iain Beaton\thanks{Corresponding author}\\
\small Department of Mathematics \& Statistics\\[-0.8ex]
\small Acadia University\\[-0.8ex] 
\small Wolfville, NS\\
\small\tt iain.beaton@acadiau.ca\\
\and
Sam Schoonhoven\\
\small Department of Mathematics \& Statistics\\[-0.8ex]
\small Acadia University\\[-0.8ex] 
\small Wolfville, NS\\
\small\tt 152186s@acadiau.ca\\
}
\date{October 31, 2024}

\begin{document}

\maketitle

\begin{abstract}
A polynomial is said to be unimodal if its coefficients are non-decreasing and then non-increasing. The domination polynomial of a graph $G$ is the generating function of the number of dominating sets of each cardinality in $G$. In \cite{IntroDomPoly2014} Alikhani and Peng conjectured that all domination polynomials are unimodal. In this paper we show that not all trees have log-concave domination polynomial. We also give non-increasing and non-decreasing segments of coefficents in trees. This allows us to show the domination polynomial trees with $\Gamma(T)-\gamma(T)<3$ are unimodal.

%\keywords{Dominating Sets  \and Reconfiguration \and Unimodality \and Average Graph Parameters.}
\end{abstract}

\setstretch{1.4}

\section{Introduction}\label{sec:intro}

A subset of vertices $S$ of a (finite, undirected) graph $G=(V,E)$ is a {\em dominating set} if and only if every vertex of $G$ is either in $S$ or adjacent to a vertex of $S$ (equivalently, $N[S]=V$ where $N[S]$ is the union of the {\em closed neighbourhoods} $N[v]$ of $v$ for all $v \in S$). As for many graph properties, one can more thoroughly examine domination via generating functions. Fro a graph $G$ or order $n$, let $d_i$ denote be the number of dominating sets of a graph $G$ of cardinality $i$. The \emph{domination polynomial} $D(G,x)$ of $G$ is defined as

$$D(G,x) = \sum_{i=0}^{n} d_ix^i.$$

\noindent We direct the reader to \cite{2012AlikhaniPHD} for a thorough discussion of domination polynomials. A natural question for any graphs polynomial is whether or not the sequence of coefficients is unimodal: a polynomial with real coefficients $a_0 + a_1x + \cdots + a_nx^n$ is said to be \emph{unimodal} if there exists $0 \leq k \leq n$, such that

$$a_0 \leq \cdots \leq a_{k-1} \leq a_k \geq a_{k-1} \geq \cdots \geq a_{n}.$$

\noindent We refer to the largest coefficient as the \emph{mode} of the sequence. Additionally, we would say the coefficient sequence has its mode at $k$. Note that under this definition, the location of the mode (i.e $k$) may not be unique. Showing a graph polynomial is unimodal has been of interest for a variety of other graph polynomials. The matching polynomial \cite{1972Heilmann,1996Krattenthaler }, the independence polynomial of claw-free graphs \cite{2007Chudnovsky, 1990Hamidoune}, the $k$-dependent set polynomial \cite{2002Horrocks} and the absolute value of the coefficients of chromatic polynomials \cite{2012Huh} have all been shown to be unimodal. In each of these papers, the polynomial was shown to be \emph{log-concave} which then implies the absolute value of the coefficients of the polynomial is unimodal. A polynomial is log-concave if for every $1 \leq i \leq n-1$, $a_i^2 \geq a_{i-1}a_{i+1}$. This is often an easier result to prove as it only considers the coefficients ``locally" and does not require the location of the mode. The unimodality of domination polynomials has been an open problem since 2009.

\vspace{4mm}

\noindent \textbf{Conjecture} \textnormal{\cite{IntroDomPoly2014}} \emph{ The domination polynomial of any graph is unimodal.}

\vspace{4mm}

\noindent The current literature on unimodality of domination polynomials is as follows. Various families of graph have been shown to have unimodal domiantion polynoimals \cite{burcroff2023unimodality, galvin2024domination, lau2022more}. Some families of graphs have even been shown to have log-concave domination polynomials \cite{2014DomFamUnimodal}. However, not all domination polynomials are log-concave as shown by a counter example on nine vertices in \cite{beaton2022unimodality}. Thus any hope to prove unimodality of domination polynomials lies in lies in showing portions of coefficient sequence of $D(G,x)$ are monotonic. For a graph $G$ with order $n$, Alikhani and Peng \cite{IntroDomPoly2014} showed the coefficients were non-decreasing (i.e. $d_{i-1} \leq d_{i}$) when $1 \leq i \leq \lceil \frac{n}{2} \rceil$. In one of the Author's Ph.D. Thesis \cite{beaton2021dominating} it was shown that for a graph without isolated vertices, the coefficients were non-increasing (i.e. $d_{i-1} \geq d_{i}$) when $\lfloor \frac{3n}{4} \rfloor \leq i \leq n$. Additionally, for large enough minimal degree $\delta(G) \geq 2ln_2(n)$ then $D(G,x)$ have been shown to be unimodal \cite{beaton2022unimodality}. Burcoff and O'Brien \cite{ burcroff2023unimodality} showed the coefficients were non-increasing when $\lfloor \frac{n+\Gamma(G)}{2} \rfloor \leq i \leq n$ where $\Gamma(G)$ is the size of the largest minimal dominating set in $G$. 

The focus of this paper will be the unimodality in domination polynomial of trees. We will rely heavily on the structure of a rooted tree. We will use the following definitions through the paper for a tree $T$ rooted at a vertex $v$. The \emph{depth} of a vertex is its distance to $v$. Additionally, for two vertices $x$ and $y$, we say $x$ is a \emph{descendant} of $y$ if $y$ is in the unique path from $x$ to the root $v$. Moreover, $x$ is a \emph{child} or \emph{grandchild} of $y$ if $x$ is a descendant of $y$ and distance one or two respectfully from $y$. Conversely, $y$ is the \emph{parent} or \emph{grandparent} of $x$ if $x$ is the child or grandchild of $y$ respectfully. Note that in a tree rooted at $v$ each vertex other than $v$ has exactly one parent.

The paper is structured as follows. In Section \ref{sec:notLC} we give an infinite family of trees with non log-concave polynomials. In Section \ref{Sec_Increase} and \ref{Sec_decrease} we show two monotonic portions of coefficients which together yield 

$$d_{\gamma(G)} \leq \cdots \leq d_{\left\lfloor\frac{n+2\gamma(T)+1}{3} \right\rfloor} \hspace{3mm}\text{ and }\hspace{3mm} d_{\left\lceil\frac{n+2\Gamma(T)-2}{3} \right\rceil} \geq \cdots \geq d_{n}.$$

\noindent This allows us to conclude that trees with $\Gamma(T)-\gamma(T)<3$ have unimodal domination polynomial. We conclude with applications to the average order of dominating sets.

\section{$D(G,x)$ is not always Log-Concave for Trees}
\label{sec:notLC}

In \cite{beaton2022unimodality} it was shown that not all domination polynomials are log-concave. This was given by a single counterexample on nine vertices, which through computation was also shown to be the smallest counterexample. In this section we give an infinite family of trees which are also not log-concave. This construction was inspired by \cite{kadrawi2023independence} which surprisingly showed that the independence polynomial of trees is not always log-concave. Consider the construction of the tree $T_k$ in Figure \ref{fig:DomPolynLC}.

\begin{center}
\begin{figure}[!h]
\begin{center}
\begin{tikzpicture}[> = stealth,semithick]
\begin{scope}[every node/.style={circle,thick,draw}]

        \node[shape=circle,draw=black,fill=white, label={[label distance=0.25mm]175:$v_0$}]  (v0) at (5.5,4) {};
    \node[shape=circle,draw=black,fill=white, label=left:$v_1$] (v1) at (2,3) {};
    \node[shape=circle,draw=black,fill=white, label=left:$v_2$] (v2) at (5.5,3) {};   
    \node[shape=circle,draw=black,fill=white, label=right:$v_3$] (v3) at (9,3) {};   
    \node[shape=circle,draw=black,fill=white, label=left:$x_{11}$] (v11a) at (0,2) {};
    \node[shape=circle,draw=black,fill=white, label=left:$y_{11}$] (v11b) at (0,1) {};
    \node[shape=circle,draw=black,fill=white, label=left:$z_{11}$] (v11c) at (0,0) {};   
    \node (v12a) at (1,2) {};
    \node (v12b) at (1,1) {};
    \node (v12c) at (1,0) {}; 
    \node (v13a) at (3,2) {};
    \node (v13b) at (3,1) {};
    \node (v13c) at (3,0) {}; 
    
    \node (v21a) at (0+4,2) {};
    \node (v21b) at (0+4,1) {};
    \node (v21c) at (0+4,0) {};   
    \node (v22a) at (1+4,2) {};
    \node (v22b) at (1+4,1) {};
    \node (v22c) at (1+4,0) {}; 
    \node (v23a) at (3+4,2) {};
    \node (v23b) at (3+4,1) {};
    \node (v23c) at (3+4,0) {}; 

    \node (v31a) at (0+8,2) {};
    \node (v31b) at (0+8,1) {};
    \node (v31c) at (0+8,0) {};   
    \node (v32a) at (1+8,2) {};
    \node (v32b) at (1+8,1) {};
    \node (v32c) at (1+8,0) {}; 
    \node[shape=circle,draw=black,fill=white, label=right:$x_{3k}$] (v33a) at (3+8,2) {};
    \node[shape=circle,draw=black,fill=white, label=right:$y_{3k}$] (v33b) at (3+8,1) {};
    \node[shape=circle,draw=black,fill=white, label=right:$z_{3k}$] (v33c) at (3+8,0) {}; 
\end{scope}

\begin{scope}
    \path [-] (v0) edge node {} (v1);
    \path [-] (v0) edge node {} (v2);
    \path [-] (v0) edge node {} (v3);
    
    \path [-] (v1) edge node {} (v11a); 
    \path [-] (v1) edge node {} (v12a);
    \path [-] (v1) edge node {} (v13a);
    \path [-] (v11a) edge node {} (v11b);
    \path [-] (v11b) edge node {} (v11c);
    \path [-] (v12a) edge node {} (v12b);
    \path [-] (v12b) edge node {} (v12c);
    \path [-] (v13a) edge node {} (v13b);
    \path [-] (v13b) edge node {} (v13c);
    
    \path [-] (v2) edge node {} (v21a); 
    \path [-] (v2) edge node {} (v22a);
    \path [-] (v2) edge node {} (v23a);
    \path [-] (v21a) edge node {} (v21b);
    \path [-] (v21b) edge node {} (v21c);
    \path [-] (v22a) edge node {} (v22b);
    \path [-] (v22b) edge node {} (v22c);
    \path [-] (v23a) edge node {} (v23b);
    \path [-] (v23b) edge node {} (v23c);
    
    \path [-] (v3) edge node {} (v31a); 
    \path [-] (v3) edge node {} (v32a);
    \path [-] (v3) edge node {} (v33a);
    \path [-] (v31a) edge node {} (v31b);
    \path [-] (v31b) edge node {} (v31c);
    \path [-] (v32a) edge node {} (v32b);
    \path [-] (v32b) edge node {} (v32c);
    \path [-] (v33a) edge node {} (v33b);
    \path [-] (v33b) edge node {} (v33c);
\end{scope}

\path (v12b) -- node[auto=false]{\ldots} (v13b);
\draw [decorate,decoration={brace,amplitude=5pt,mirror,raise=2ex}](0,0) -- (3,0) node[midway,yshift=-2em]{$k$};

\path (v22b) -- node[auto=false]{\ldots} (v23b);
\draw [decorate,decoration={brace,amplitude=5pt,mirror,raise=2ex}](0+4,0) -- (3+4,0) node[midway,yshift=-2em]{$k$};

\path (v32b) -- node[auto=false]{\ldots} (v33b);
\draw [decorate,decoration={brace,amplitude=5pt,mirror,raise=2ex}](0+8,0) -- (3+8,0) node[midway,yshift=-2em]{$k$};
\end{tikzpicture}
\end{center}
\caption{The tree $T_k$}%
\label{fig:DomPolynLC}%
\end{figure}
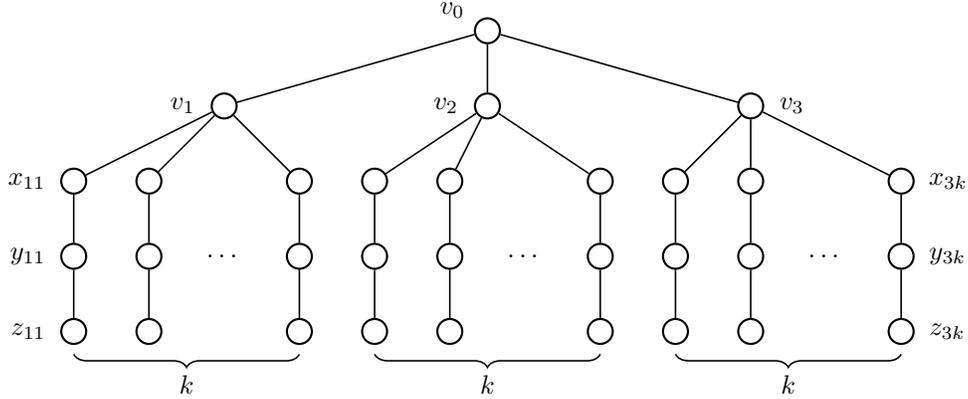
\end{center}

\noindent We will now show that $D(T_k,x)$ is not log-concave for sufficiently large $k$.

\begin{proposition}
\label{prop:NotLC}
When $k \geq 4$ then $D(T_k,x)$ is not log-concave.
\end{proposition}

\begin{proof}
Consider the vertices of $T_k$ as labelled in Figure \ref{fig:DomPolynLC}.
Note that $T_k$ is rooted at $v_0$ and vertices $x_{ij}$, $y_{ij}$, and $z_{ij}$ are at depth 2, 3, and 4 respectively for all $1 \leq i \leq 3$ and $1 \leq j \leq k$.
Additionally vertex $x_{ij}$, $y_{ij}$, and $z_{ij}$ are all descendants of $v_i$ and induce the $j^{th}$ 3-path which lies below $v_i$.
For each $i$, let $X_i$, $Y_i$ and $Z_i$ denote the set of all $x_{ij}$, $y_{ij}$, and $x_{ij}$ respectively.
Note that $|X_i|=|Y_i|=|Y_i|=k$.
Additionally note that $T_k$ has $9k+4$ vertices and $\gamma(T_k)=3k+1$.

The set $Y_1 \cup Y_2 \cup Y_3 \cup \{v_0\}$ is the only dominating set of size $3k+1$.
Let $S_\gamma$ denote this unique minimum dominating set.
Thus $d_{\gamma(T_k)} = 1$.
We will now consider $d_{\gamma(T_k)+1}$ and $d_{\gamma(T_k)+2}$.

First consider dominating sets of size $\gamma(T_k)+1$.
The non-minimal dominating sets are formed by adding a vertex any $v \notin S_\gamma$ to $S_\gamma$.
Thus there are $6K+3$ non-minimal dominating sets of size $\gamma(T_k)+1$.
The minimal dominating sets of size $\gamma(T_k)+1$ take on two forms.
The first form is $S_\gamma \cup \{x_{ij}, z_{ij}\} - \{y_{ij}\}$ for any $y_{ij}$.
There are $3k$ such minimal dominating sets.
The second form is $S_\gamma \cup \{v_i\} \cup Z' - Y'$ for each $i$ where $Z'$ is non-empty any subset $Z' \subseteq Z_i$ and $Y'$ is the corresponding neighbours of $Z'$ in $Y$.
There are $3(2^k-1)$ minimal dominating sets of this form.
Altogether we have $d_{\gamma(T_k)+1}=9k+3+3(2^k-1)=9k+3 \cdot 2^k$.

It suffices to only give a lower bound on $d_{\gamma(T_k)+2}$.
Let $S$ be a dominating set of  size $\gamma(T_k)+2$.
We will consider three disjoint cases based on how many vertices of $\{v_1, v_2, v_3\}$ are in $S$.
For a given $i$, let $ZY_i$ denote the collection of $k$-sets which contain exactly on vertex from $\{y_{ij}, z_{ij}\}$ for all $1 \leq j \leq k$.
Note that $|ZY_i|=2^k$ for each $i$.

\vspace{2mm}

\noindent \underline{Case 1:} $S$ contains exactly one vertex in $\{v_1, v_2, v_3\}$.
Without loss of generality let $v_1 \in S$.
If $v_0 \in S$, then $S=\{v_0, v_1, w\} \cup T_1 \cup Y_2 \cup Y_3$ is a dominating set where $T_1 \in ZY_1$ and $w \in X_1 \cup X_2 \cup X_3\cup Z_2 \cup Z_3$.
There are $5k \cdot 2^k$ such dominating sets.
If $v_0 \notin S$, then $S=\{v_1, x_2, x_3\} \cup T_1 \cup Y_2 \cup Y_3$ is a dominating set where $T_1 \in ZY_1$ and $x_2 \in X_2$ and $x_3 \in X_3$.
There are $k^22^k$ such dominating sets.
There are three ways to choose one vertex from $\{v_1, v_2, v_3\}$.
Thus by symmetry that gives at least $3(5k+k^2) \cdot 2^k$ dominating sets containing exactly one vertex from $\{v_1, v_2, v_3\}$.

\vspace{2mm}

\noindent \underline{Case 2:} $S$ contains exactly two vertices in $\{v_1, v_2, v_3\}$.
Without loss of generality let $v_1, v_2 \in S$.
Then $S=\{w, v_1, v_2\} \cup T_1 \cup T_2 \cup Y_3$ is a dominating set where $T_1 \in ZY_1$, $T_2 \in ZY_2$, and $w \in \{v_0\} \cup X_3$.
There are three ways to choose two vertices from $\{v_1, v_2, v_3\}$.
Thus there are at least $3(k+1)2^{2k}$ dominating sets which contain exactly two vertices in $\{v_1, v_2, v_3\}$

\vspace{2mm}

\noindent \underline{Case 3:} $S$ contains all three vertices in $\{v_1, v_2, v_3\}$.
Then $S=\{v_1, v_2, v_3\} \cup T_1 \cup T_2 \cup T_3$ is a dominating set where $T_1 \in ZY_1$, $T_2 \in ZY_2$, and $T_3 \in ZY_3$.
Thus there are at least $2^{3k}$ dominating sets which contain exactly two vertices in $\{v_1, v_2, v_3\}$

\vspace{2mm} 

\noindent Altogether that gives the lower bound $d_{\gamma(T_k)+2} \geq 3(5k+k^2) \cdot 2^k + 3(k+1)2^{2k} + 2^{3k}$.

Finally we will show $d_{\gamma(T_k)+2}d_{\gamma(T_k)}>(d_{\gamma(T_k)+1})^2$ when $k \geq 4$. 
Note that 

\vspace{-2mm}

\begin{align*}
d_{\gamma(T_k)+2}d_{\gamma(T_k)} &\geq 3(5k+k^2) \cdot 2^k + 3(k+1)2^{2k} + 2^{3k}, \text{and} \\
(d_{\gamma(T_k)+1})^2  &= 81k^2+54k \cdot 2^k + 9 \cdot 2^{2k}.
\end{align*}

\noindent Note that for all $k\geq 4$ then $6 \cdot 2^{2k} > 81k^2$ thus

\vspace{-2mm}

\begin{align*}
d_{\gamma(T_k)+2}d_{\gamma(T_k)} - (d_{\gamma(T_k)+1})^2 &\geq 2^{3k} + 3(k-2)2^{2k} +(3k^2-39k)2^k - 81k^2 \\
  &>2^{3k}+ 3(k-2)2^{2k} +(3k^2-39k)2^k-6 \cdot 2^{2k} \\
  &= 3(k-4)2^{2k} +(2^{2k}+3k^2-39k)2^k
\end{align*}

\noindent For all $k\geq 4$ we have $k-4 \geq 0$ and $2^{2k}+3k^2-39k \geq 0$ therefore $d_{\gamma(T_k)+2}d_{\gamma(T_k)} - (d_{\gamma(T_k)+1})^2 >0$ and $D(T_k,x)$ is not log-conave for all $k \geq 4$.

\end{proof}

\section{Non-decreasing Segment of Coefficients}
\label{Sec_Increase}

In this section we will show that

$$d_{\gamma(T)} \leq d_{\gamma(T)+1} \leq  \cdots \leq d_{\left\lfloor\frac{n+2\gamma(T)+1}{3} \right\rfloor},$$

\noindent where $T$ is a tree of order $n$. We first begin with some useful definitions from \cite{2021Beaton} which categorize vertices according to a  dominating set $S$. For a graph $G$, let $\mathcal{D}(G)$ denotes the collection of dominating sets of $G$. For a dominating set $S$ of $G$ let

$$a(S) = \{v \in S : S-v \notin \mathcal{D}(G)\},$$

\noindent denote the set of \emph{critical} vertices of $S$ with respect to domination (in that their removal makes the set no longer dominating). Note that a dominating set $S$ is minimal if and only if $S=a(S)$. To contrast critical vertices, we say a vertex in $S$ is \emph{supported} if it is not critical. That is, $v$ is a supported vertex of $S$ if $v \in S-a(S)$. We say a supported vertex $v \in S - a(S)$ is \emph{supported by} $u \in S$ if $N[u] \cap N[v] \neq \emptyset$. This brings us to our first observation regarding critical vertices.

\begin{observation}
\label{obs:IsoImpliesCrit}
 Let $G$ be a graph with dominating set $S$ containing $v$. If $v$ has no neighbours in $S$ then $v \in a(S)$.
\end{observation}

\noindent As it turns out, the difference between $d_i$ and $d_{i-1}$ depends on the total number of critical vertices in dominating sets of those sizes.

\begin{lemma}
\label{lem:aidi}
\textnormal{\cite{beaton2021dominating}} For a graph $G$ with $n$ vertices.
  
$$a(G,i)=\sum\limits_{S \in \mathcal{D}_i(G)}|a(S)| = id_i(G)-(n-i+1)d_{i-1}(G),$$

\noindent where $\mathcal{D}_i(G)$ denotes the collection of dominating sets of size $i$.
\end{lemma}

\begin{observation}
\label{obs:adi}
\textnormal{\cite{beaton2021dominating}} For a graph $G$ on $n$ vertices $d_i(G) \leq d_{i-1}(G)$ if and only if $a(G,i) \leq (2i-n-1)d_{i}(G) $.
\end{observation}

%The tree family of graphs have structural properties that give us a stable foundation to build off of, such as the existence of parent and child vertices when any such graph is rooted. Trees are defined as undirected graphs in which any two vertices are connected by exactly one path. This is equivalent to saying the graph is connected and acyclic. Two important components to trees helpful in characterizing them are their leaves and stems. 
%
%\noindent Define the set of leaves in a tree $T$ as $$L(T)=\{v\in V(T): deg(v)=1\}$$ Define the set of stems in a tree $T$ as $$S(T)=\{v\in V(T):u\in L(T),N(u)=\{v\}\}$$There is significant theoretical infrastructure built around trees, like the following theorem.
%
%\begin{theorem}\cite{MagdalenaLemanska2004}
%\label{thm:dombound}
%If T is a tree with $n\ge 3$ vertices and $\ell$ leaves then $\gamma(T)\ge\frac{n}{3}-\frac{\ell}{3}+\frac{2}{3}$.
%\end{theorem}
%
%\noindent This theorem bounding the domination number of a tree will aide us bounding the monotonic portion of a tree's domination sequence. In doing this we will see that a bound for $a(T,i)$ is needed, and thus $|a(D)|$ where $D\in\mathcal{D}(T)$. We find a bound for $|a(D)|$ with the following lemma.

We will now work towards a lemma which bounds $a(S)$ for any dominating set $S$ of a tree $T$. Clearly for any minimal dominating set $M$, we have that $a(M)=M$. Moreover $\gamma(T) \leq a(M)$. As we add vertices to $M$ to create other dominating sets  $S$, the question becomes how do this lower bound change? It turns out that the lower bound decreases by 1 for each additional vertex above $\gamma(T)$ and $\Gamma(T)$ respectively.

\begin{lemma}
\label{lem_a(S)lowerbound}
If T is a tree then for any dominating set $S \in \mathcal{D}(T)$ then $|a(S)| \geq 2\gamma(T)-|S|$.
\end{lemma}

\begin{proof}
For simplicity denote $\gamma (T)$ by $\gamma$.
Let $t \geq 0$ be the integer such that $|S|= \gamma+t$.
It then suffices to show $|a(S)| \geq \gamma-t$.
To do this we will induct on $t\ge0$.
For $t=0$, $|S|=\gamma$ and $|a(S)|=\gamma$ so our base case is satisfied.
Suppose our statement holds for $0\le t\le k-1$, $k\in\mathbb{N}$.
Then, for $t=k$ we have $|S|=\gamma+k$ and we wish to show that $|a(S)|\ge\gamma-k$.
First note that if $a(S)$ has no supported vertices then $|a(S)|=|S|\ge\gamma-k$.
Thus we may assume that $S$ contains at least one supported vertex.
For any supported vertex $v\in S$ we have that $|S-v|=\gamma+k-1$. 
So by our inductive hypothesis $$|a(S-v)|\ge\gamma-(k-1)=\gamma-k+1.$$
Thus, it suffices to find a supported vertex $v$ such that $|a(S)|+1\ge|a(S-v)|$.
That is, the removal of $v$ from $S$ creates at most $1$ more critical vertex.
To find this $v$, root $T$ at any vertex and choose $v$ to be a supported vertex with maximum depth.
Now, suppose that $v$ does not satisfy our condition such that $|a(S)|+2\le|a(S-v)|$.
That is, the removal of $v$ from $S$ creates at least $2$ more critical vertices.
Note that $a(S)\subseteq a(S-v)$ so let $\{x_1,x_2\}\subseteq a(S-v)-a(S)$ be our new critical vertices.
Also note that $x_1$ and $x_2$ were supported vertices in $S$ with $x_1,x_2\neq v$ and $x_1\neq x_2$.
%As $v$ was supported in $S$, every vertex in $N[v]$ is doubly dominated.
Let $u_1$ and $u_2$ be the vertices which only have $x_1$ and $x_2$ as neighbours respectively in $S-v$.
Therefore 

$$N[u_1]\cap S=\{x_1,v\}\text{ and }N[u_2]\cap S=\{x_2,v\}.$$

\noindent Note $u_1\neq u_2$ otherwise $u_1$ would be dominated by both $x_1$ and $x_2$ in $S-v$.
%In $S-v$, $u_1$ and $u_2$ only have $x_1$ and $x_2$ as their neighbours respectively.
As $x_1$ and $x_2$  were supported vertices in $S$ then they are no deeper in $T$ than $v$.
Thus any common neighbour of $v$ and $x_1$ or $x_2$ must be the parent of $v$.
Hence $u_1$ and $u_2$ are both the parent of $v$ so $u_1= u_2$ which contradicts $u_1\neq u_2$. 
\end{proof}

The previous lemma, together with Observation \ref{obs:adi} give us our main result in this section regarding domination polynomials.

\begin{theorem}
\label{thm:increasing}
Let $T$ be a tree of order $n$. Then

$$d_{\gamma(T)} \leq d_{\gamma(T)+1} \leq \cdots \leq d_{\left\lfloor\frac{n+2\gamma(T)+1}{3} \right\rfloor}.$$

\noindent where $d_i$ denotes the number of dominating sets in $T$ of size $i$.
\end{theorem}

\begin{proof}
For any dominating set $S$ it follows from Lemma \ref{lem_a(S)lowerbound} that $2\gamma(T) - |S| \leq a(S)$. Therefore 

$$(2\gamma(T) - i)d_i \leq a(T,i)$$

\noindent By Observation \ref{obs:adi}, $d_{i-1} \leq d_{i}$ if and only if $a(T,i) \geq (2i-n-1)d_i$.
Thus $d_{i-1} \leq d_i$ when $i \leq \frac{n+2\gamma(T)+1}{3}$.
As $i$ must be an integer, we obtain our result.
\end{proof}

%%%%%%%%%%%%%%%%%%%%%%%%%%%%%%%%%%%%%%%%%%%%%%%%%%%%%%%%%%%%%%%%%%%%%%%%%%%%%%%%%%%%%%%%%%%%%%%%
\section{Non-increasing Segment of Coefficients}
\label{Sec_decrease}
%%%%%%%%%%%%%%%%%%%%%%%%%%%%%%%%%%%%%%%%%%%%%%%%%%%%%%%%%%%%%%%%%%%%%%%%%%%%%%%%%%%%%%%%%%%%%%%%

\noindent In this section we will show 

$$ d_{\left\lceil\frac{n+2\Gamma(T)-2}{3} \right\rceil} \geq \cdots \geq d_{n-1} \geq d_{n},$$

\noindent were $T$ is a tree of order $n$. We will again investigate how the domination critical vertices $a(S)$ behaviour when adding vertices to $S$. However, now we seek the upperbound $a(S) \leq 2\Gamma(T)-|S|$. This upperbound has proven to be significantly harder to prove than the lowerbound in the previous section. The essence of the proof remains the same. If for some dominating set $S$ we have $a(S) < 2\Gamma(T)-|S|$, then we can find a minimal dominating set larger than $\Gamma(T)$. Constructing these larger minimal dominating sets requires algorithms which iterate the reconfiguration rules introduced in \cite{2024BeatonReConfig}. Although we find the algorithms interesting in their own right, they are certainly very tedious for the objective of this section.

We will begin by expanding our definition regarding domination critical vertices. First partition the vertices not in $S$ into the following two sets:

\vspace{-6mm}

\begin{align*}
N_1(S) &= \{v \in V-S : |N[v] \cap S| = 1 \}\\
N_2(S) &= \{v \in V-S : |N[v] \cap S| \geq 2 \}.
\end{align*}

\noindent Furthermore, recall the partition $a(S) = a_1(S) \cup a_2(S)$, where 

\vspace{-6mm}

\begin{align*}
a_1(S) &= \{v \in a(S) : N[v] \cap N_1(S) \neq \emptyset \}\\
a_2(S) &= \{v \in a(S) : N[v] \cap N_1(S) = \emptyset \}.
\end{align*}

\noindent Partitioning these two sets lends the following results.

\begin{lemma}\cite{2021Beaton}
\label{lem:a1n1}
 Let $G$ be a graph. For any dominating set $S \in D_G$, $|a_1(S)| \leq |N_1(S)|$.
\end{lemma} 

\begin{observation}
\label{obs:a2ImpliesIso}
 Let $G$ be a graph. For any dominating set $S$ if $v \in a_2(S)$ then $N(v) \subseteq N_2(S)$ and hence has no neighbours in $S$.
\end{observation}

\begin{proof}
Let $v \in a_2(S)$.
As $v$ is critical then $S-v$ is not a dominating set.
As $S$ was a dominating set then any vertex not dominated by $S-v$ is in $N[v]$.
More specifically, either $v$ or some vertex in $N(v)-S$ is not dominated by $S-v$.
By the definition of $a_2(S)$, $N[v] \cap N_1(S) = \emptyset$.
Therefore every vertex in $N(v)-S$ must be in $N_2(S)$, and thus is still dominated by its other neighbour in $S$.
Thus, $v$ is not dominated by $S-v$ and hence it has no neighbours in $S-v$.
Therefore  $N(v) \subseteq N_2(S)$. \hfill\nopagebreak$\Box$\par\bigskip
\end{proof}

\noindent We now introduce an algorithm which will be used to reconfigure minimal dominating sets.

\begin{center}
\begin{algorithm}[H]
\label{alg:MakeMinimal}
Input a finite tree $T$ rooted at a given vertex $v$  \;
Input a dominating set $M_0$ of $T$ \;
Set $i=0$  \;
  \While{$M_i$ is not a minimal dominating set}
  {
   Choose a supported vertex $u_i \in M_i - a(M_i)$ of least depth\;
   Set $A_{i+1}$ to be all vertices in $a_1(M_i)$ which neighbour $u_i$\;
   Set $N_{i+1}$ to be all vertices in $N_1(M_i)$ which neighbour a vertex in $A_{i+1}$\;
   Set $M_{i+1} = (M_{i}-A_{i+1}) \cup N_{i+1}$\;
   Set $i=i+1$  \;
  }
Return $M_i$  \;
 \caption{Creates a Minimal Dominating set $M_i$ with $|M_i| \geq |M_0|$}
\end{algorithm}
\end{center}

Consider an example of Algorithm \ref{alg:MakeMinimal} illustrated in Figure \ref{fig:AlgExample} for a tree on 17 vertices rooted at the vertex $v$. 
The set $M_0$ is the vertices shaded black in Figure \ref{fig:AlgExample} $(a)$. 
For $i=0$, there is exactly one supported vertex $u_0 \in M_0$. 
The lone neighbour of $u_0$ in $M_0$ is also in $a_1(M_0)$ and forms $A_1$.
The one vertex in $A_1$ has two neighbours in $N_1(M_{0})$ (its middle and right child) which form $N_1$.
Let $M_1=(M_0-A_1) \cup N_1$ which is featured in Figure \ref{fig:AlgExample} $(b)$.
Now set $i=1$ and we continue as $M_1$ is not a minimal dominating set.
There is exactly one supported vertex $u_1 \in M_1$. 
It has two neighbours in $a_1(M_0)$ which together form $A_2$.
Each vertex in $A_2$ has one neighbour in $N_1(M_{1})$ which together form $N_2$.
Let $M_2=(M_1-A_2) \cup N_2$ which is featured in Figure \ref{fig:AlgExample} $(c)$.
$M_2$ is a minimal dominating set and thus Algorithm \ref{alg:MakeMinimal} returns $M_2$ which is larger than $M_0$.

\begin{figure}[!h]
\def\c{1}
\centering
\subfigure[$M_0$]{
\scalebox{\c}{
\begin{tikzpicture}
\begin{scope}[every node/.style={circle,thick,draw}]
    \node[shape=circle,draw=black,fill=white, label=above:$v$] (1) at (0,0) {};
    \node[shape=circle,draw=black,fill=black, label={[label distance=0.75mm]45:$u_0$}] (2) at (0.5,-0.5) {};
    \node[shape=circle,draw=black,fill=black] (3) at (1,-1) {};
    \node[shape=circle,draw=black,fill=white] (4) at (1.5,-1.5) {};
    \node[shape=circle,draw=black,fill=white] (5) at (2,-2) {};
    \node[shape=circle,draw=black,fill=black] (6) at (2.5,-2.5) {};
    \node[shape=circle,draw=black,fill=black] (7) at (3,-3) {};
    \node[shape=circle,draw=black,fill=white] (8) at (3.5,-3.5) {};
    \node[shape=circle,draw=black,fill=black] (9) at (2,-3) {};
    \node[shape=circle,draw=black,fill=white] (10) at (1.5,-3.5) {};
    
    \node[shape=circle,draw=black,fill=white] (11) at (0.5,-1.5) {};
    \node[shape=circle,draw=black,fill=black] (12) at (0,-2) {};
    \node[shape=circle,draw=black,fill=white] (13) at (-0.5,-2.5) {};

    \node[shape=circle,draw=black,fill=white] (14) at (1,-1.5) {};
    \node[shape=circle,draw=black,fill=white] (15) at (1,-2) {};
    \node[shape=circle,draw=black,fill=black] (16) at (1,-2.5) {};
      
    \node[shape=circle,draw=black,fill=black] (17) at (-0.5,-0.5) {};
\end{scope}

\begin{scope}
    \path [-] (1) edge node {} (2);
    \path [-] (2) edge node {} (3);
    \path [-] (3) edge node {} (4);
    \path [-] (4) edge node {} (5);
    \path [-] (5) edge node {} (6);
    \path [-] (6) edge node {} (7);
    \path [-] (7) edge node {} (8);
    
    \path [-] (6) edge node {} (9);
    \path [-] (9) edge node {} (10);

    \path [-] (3) edge node {} (11);
    \path [-] (11) edge node {} (12);
    \path [-] (12) edge node {} (13);
    
    \path [-] (3) edge node {} (14);
    \path [-] (14) edge node {} (15);
    \path [-] (15) edge node {} (16);
    
    \path [-] (1) edge node {} (17);
\end{scope}
\end{tikzpicture}}}
\qquad
\subfigure[$M_1$]{
\scalebox{\c}{
\begin{tikzpicture}
\begin{scope}[every node/.style={circle,thick,draw}]
    \node[shape=circle,draw=black,fill=white, label=above:$v$] (1) at (0,0) {};
    \node[shape=circle,draw=black,fill=black, label={[label distance=0.75mm]45:$u_0$}] (2) at (0.5,-0.5) {};
    \node[shape=circle,draw=black,fill=white] (3) at (1,-1) {};
    \node[shape=circle,draw=black,fill=black] (4) at (1.5,-1.5) {};
    \node[shape=circle,draw=black,fill=white] (5) at (2,-2) {};
    \node[shape=circle,draw=black,fill=black,label={[label distance=0.75mm]45:$u_1$}] (6) at (2.5,-2.5) {};
    \node[shape=circle,draw=black,fill=black] (7) at (3,-3) {};
    \node[shape=circle,draw=black,fill=white] (8) at (3.5,-3.5) {};
    \node[shape=circle,draw=black,fill=black] (9) at (2,-3) {};
    \node[shape=circle,draw=black,fill=white] (10) at (1.5,-3.5) {};
    
    \node[shape=circle,draw=black,fill=white] (11) at (0.5,-1.5) {};
    \node[shape=circle,draw=black,fill=black] (12) at (0,-2) {};
    \node[shape=circle,draw=black,fill=white] (13) at (-0.5,-2.5) {};
    
    \node[shape=circle,draw=black,fill=black] (14) at (1,-1.5) {};
    \node[shape=circle,draw=black,fill=white] (15) at (1,-2) {};
    \node[shape=circle,draw=black,fill=black] (16) at (1,-2.5) {};

    \node[shape=circle,draw=black,fill=black] (17) at (-0.5,-0.5) {};
\end{scope}

\begin{scope}
    \path [-] (1) edge node {} (2);
    \path [-] (2) edge node {} (3);
    \path [-] (3) edge node {} (4);
    \path [-] (4) edge node {} (5);
    \path [-] (5) edge node {} (6);
    \path [-] (6) edge node {} (7);
    \path [-] (7) edge node {} (8);
    
    \path [-] (6) edge node {} (9);
    \path [-] (9) edge node {} (10);

    \path [-] (3) edge node {} (11);
    \path [-] (11) edge node {} (12);
    \path [-] (12) edge node {} (13);
    
    \path [-] (3) edge node {} (14);
    \path [-] (14) edge node {} (15);
    \path [-] (15) edge node {} (16);
    
    \path [-] (1) edge node {} (17);
\end{scope}
\end{tikzpicture}}}
\qquad
\subfigure[$M_2$]{
\scalebox{\c}{
\begin{tikzpicture}
\begin{scope}[every node/.style={circle,thick,draw}]
    \node[shape=circle,draw=black,fill=white, label=above:$v$] (1) at (0,0) {};
    \node[shape=circle,draw=black,fill=black, label={[label distance=0.75mm]45:$u_0$}] (2) at (0.5,-0.5) {};
    \node[shape=circle,draw=black,fill=white] (3) at (1,-1) {};
    \node[shape=circle,draw=black,fill=black] (4) at (1.5,-1.5) {};
    \node[shape=circle,draw=black,fill=white] (5) at (2,-2) {};
    \node[shape=circle,draw=black,fill=black,label={[label distance=0.75mm]45:$u_1$}] (6) at (2.5,-2.5) {};
    \node[shape=circle,draw=black,fill=white] (7) at (3,-3) {};
    \node[shape=circle,draw=black,fill=black] (8) at (3.5,-3.5) {};
    \node[shape=circle,draw=black,fill=white] (9) at (2,-3) {};
    \node[shape=circle,draw=black,fill=black] (10) at (1.5,-3.5) {};
    
    \node[shape=circle,draw=black,fill=white] (11) at (0.5,-1.5) {};
    \node[shape=circle,draw=black,fill=black] (12) at (0,-2) {};
    \node[shape=circle,draw=black,fill=white] (13) at (-0.5,-2.5) {};

    \node[shape=circle,draw=black,fill=black] (14) at (1,-1.5) {};
    \node[shape=circle,draw=black,fill=white] (15) at (1,-2) {};
    \node[shape=circle,draw=black,fill=black] (16) at (1,-2.5) {};

    \node[shape=circle,draw=black,fill=black] (17) at (-0.5,-0.5) {};
\end{scope}

\begin{scope}
    \path [-] (1) edge node {} (2);
    \path [-] (2) edge node {} (3);
    \path [-] (3) edge node {} (4);
    \path [-] (4) edge node {} (5);
    \path [-] (5) edge node {} (6);
    \path [-] (6) edge node {} (7);
    \path [-] (7) edge node {} (8);
    
    \path [-] (6) edge node {} (9);
    \path [-] (9) edge node {} (10);

    \path [-] (3) edge node {} (11);
    \path [-] (11) edge node {} (12);
    \path [-] (12) edge node {} (13);
    
    \path [-] (3) edge node {} (14);
    \path [-] (14) edge node {} (15);
    \path [-] (15) edge node {} (16);
    
    \path [-] (1) edge node {} (17);
\end{scope}
\end{tikzpicture}}}

\caption{Example of Algorithm \ref{alg:MakeMinimal}}%
\label{fig:AlgExample}%
\end{figure}

Throughout the application of Algorithm \ref{alg:MakeMinimal}, we repeatedly remove vertices from $a_1(M_i)$ and replace then with all of their neighbours in $N_1(M_i)$. The addition of new vertices potentially creates new supported vertices at greater depth. For example, $u_1$ was critical in $M_0$, however, $u_1$ was supported in $M_1$. For this reason it is not clear that Algorithm \ref{alg:MakeMinimal} will produce a minimal dominating set. To show this we will need the following Lemma which gives some general properties which hold when adding and removing these particular vertices.

%\begin{lemma}
%\label{lem:supportdistance}
%Let $S$ be a dominating set of a tree $T$ with $x \in a(S)$. Suppose $x$ becomes supported in $S \cup \T$ for some $T \subseteq S-V$. Then
%
%\begin{itemize}
%\item[$(i)$] $d(x,u)=1$ if $x \in a_2(S)$, or 
%\item[$(ii)$] $d(x,u)=2$ if $x \in a_1(S)$. 
%\end{itemize}
%\end{lemma}
%
%\begin{proof}
%Note that $v$ is the only vertex added to 
%\end{proof}

\begin{lemma}
\label{lem:move_a1_to_n1}
Let $S$ be a dominating set of $T$ rooted at vertex $v$ with $A \subseteq a_1(S)$ and $N=N(A) \cap  N_1(S)$. Set $S' = (S-A) \cup N$. If every vertex in $A$ has the same depth in $T$ then
\begin{itemize}
\item[$(i)$] $|S'| \geq |S|$
\item[$(ii)$] The vertices in $A$ along with every descendant of the vertices in $A$ are dominated in $S'$.

\end{itemize}
Additionally, if $S'$ is a dominating set and every vertex of $N$ is a child of some vertex in $A$ then 
\begin{itemize}

\item[$(iii)$] No vertex in $N$ is adjacent to any other vertex in $S'$, hence $N \subseteq a(S')$.
\item[$(iv)$] Any vertex $x \in a(S)$ which is now supported in $S'$ is a grandchild of some vertex in $N$ and does not have its parent in $S'$.
\end{itemize}
\end{lemma}

\begin{proof}
Let $S$ be a dominating set of $T$ rooted at vertex $v$ with $A \subseteq a_1(S)$ and $N=N(A) \cap  N_1(S)$.
Suppose every vertex in $A$ has the same depth in $T$ and let $S' = (S-A) \cup N$.

$(i)$ By definition of $a_1(S)$ and $N_1(S)$, each vertex in $N_1(S)$ has exactly one neighbour in $a_1(S)$.
Therefore each vertex in $N$ has exactly one neighbour in $A$ so $|A| \leq |N|$.
Moreover  $|S'| \geq |S|$.

$(ii)$ Recall that $N=N(A) \cap  N_1(S)$.
Therefore every vertex in $A$ has a neighbour in $N$.
Thus the vertices in $A$ are dominated in $S'$.
Now suppose some descendant $u$ of $A$ is not dominated in $S'$.
Since $S$ was a dominating set, then $u$ was previously dominated by $S$.
More specifically, $u$ was dominated by the vertices in $A$ and $N[u] \cap S \subseteq A$
As $u$ is a descendant of $A$, then $u$ must be a child of some vertex in $A$.
Thus $v \in N_1(S)$ and hence $u \in N$ which contradicts the fact that $u$ is not dominated in $S'$.

For $(iii)$ and $(iv)$ suppose that $S'$ is a dominating set and every vertex of $N$ is a child of some vertex in $A$.
Note that every vertex in $A$ has the same depth so every vertex $N$ must also have the same depth.

$(iii)$ Recall that $N \subseteq N_1(S)$.
Therefore each vertex in $N$ had exactly one neighbour in $S$ which was in $A$.
As $A$ was removed in $S'$ then $N$ has no neighbours in $S'$ other than possibly those in $N$.
However, every vertex in $N$ has the same depth.
So no vertex in $N$ is adjacent to any other vertex in $S'$.
Thus by Observation \ref{obs:IsoImpliesCrit}, $N \subseteq a(S')$.

$(iv)$ let $x \in a(S)$ which is now supported in $S'$.
First note, $x \in S$ and $x \in S'$ so $x \notin A$ and $x \notin N$.
The addition of $N$ caused $x$ to no longer be domination critical. 
Therefore $x$ must be supported by some vertex in $N$ and hence distant at most two away from some vertex in $N$.
Each vertex in $N$ previously had exactly one neighbour in $S$, which was in $A$ and hence removed in $S'$. 
Thus $x$ adjacent to any vertex in $N$.
Moreover, $x$ must be distance exactly two from a vertex in $N$.
Let $x'$ be the vertex in $N$ which is distance two from $x$.
Moreover let $y'$ be the parent of $x'$ which is necessarily in $A$.
As $T$ is a tree then $x$ either has depth two less than $x'$, the same depth as $x'$, or $x$ has depth two greater than $x'$. 

Suppose $x$ has depth two less than $x'$ and $x$ then $x$ is the grandparent of $x'$ and hence the parent of $y'$.
Moreover, $x$ has depth two less than every vertex in $N$.
As $x \in a(S)$ and adjacent to $y' \in S$ then it follows from the contrapositive of Observation \ref{obs:a2ImpliesIso} that $x \in a_1(S)$.
Therefore there exists a vertex $z \in N_1(S)$ whose lone neighbour in $S$ was $x$.
We will now show that $z \in N_1(S')$ which implies $x \in a_1(S')$.
Suppose not, that is suppose $z \notin N_1(S')$.
Note that $z \notin N$ as $x \notin A$ so $z \notin S'$.
Therefore $z \in N_2(S')$ and now has at least two neighbours in $S'$.
One neighbour must be $x$, and then other(s) must be in $N$.
Therefore $x$ is the parent of $z$ and $z$ is the parent of at least one vertex in $N$.
Each parent of a vertex in $N$ is necessarily in $A$.
However, $A \subseteq S$ and $z \notin S$ so $z \notin A$.
This forms a contradiction and so $z \in N_1(S')$.
As $x \in S'$ then $x$ is the lone neighbour of $z$ is $S'$ and hence $x \in a_1(S')$.
Thus $x \notin S'-a(S')$ and hence is not supported in $S'$.
Therefore the case where $x$ has depth two less than $x'$ is impossible.

Now suppose, $x$ has the same depth as $x'$. 
As $T$ is a tree and $x$ and $x'$ are distance two from each other, then they must share the same parent $y'$.
Recall that $y' \in A$ and hence was previously in $S$.
Therefore as $x \in a(S)$ and it had a neighbour (its parent $y'$) in $S$, then it follows from Observation \ref{obs:a2ImpliesIso} that $x \in a_1(S)$.
Therefore there exists a vertex $z \in N_1(S)$ whose lone neighbour in $S$ was $x$.
Moreover, $z$ must be a child of $x$ and hence at depth one lower than any vertex in $N$.
Thus $z \notin N$ nor is $z$ is adjacent to any vertex in $N$ as its parent is $x$.
Therefore $z \in N_1(S')$ and its lone neighbour in $S'$ is still $x$.
Hence $x \in a_1(S')$, and it follows that $x$ can not have depth equal to $x'$.
Thus $x$ must be have depth two greater than $x'$ and hence $x$ is a grandchild of some vertex in $N$.
From $(iii)$ we have that a child of any vertex in $N$ is not in $S'$.
Thus the parent of $x$, which is a child to some vertex in $N$, is not in $S'$. \hfill\nopagebreak$\Box$\par\bigskip
\end{proof}

The results from Lemma \ref{lem:move_a1_to_n1} will be useful in proving which initial dominating sets $M_0$ allow Algorithm \ref{alg:MakeMinimal} to terminate. Additionally, we will use Lemma \ref{lem:move_a1_to_n1} later in this paper to construct the initial dominating sets $M_0$ from other minimal dominating sets $M$.

\begin{theorem}
\label{thm:algterm}

Let $T$ be a finite tree rooted at a vertex $v$ and $M_0$ a dominating set of $T$. If 

\begin{itemize}
\item[$(a)$] Every parent of a supported vertex is not in $M_0$, and
\item[$(b)$] No supported vertex is a descendant of another supported vertex.
\end{itemize}

\noindent Then Algorithm \ref{alg:MakeMinimal} outputs a minimal dominating set $M_i$ with $|M_i| \geq |M_0|$.
\end{theorem}

\begin{proof}
Let $T$ be a finite tree rooted at a vertex $v$ and $M_0$ a dominating set which satisfies conditions $(a)$ and $(b)$.
If $M_0$ is a minimal dominating set then the algorithm terminates.
So suppose $M_0$ is not a minimal dominating set.
Let $u_0 \in M_0 - a(M_0)$ be the supported vertex of least depth. 
Now set $A_1$ to be the set of all vertices in $M_0$ which neighbour $u_0$. 
We define $A_1$ this way so that the removal of $A_1$ from $M_0$ makes $u_0$ a domination critical vertex.
We will now show $A_1$ is also all vertices in $a_1(M_0)$ which neighbour $u_0$.
By the contrapositive of Observation \ref{obs:IsoImpliesCrit}, as $u_0 \in M_0- a(M_0)$ then $u_0$ must have a neighbour in $M_0$.
More specifically $A_1$ is not empty. 
By condition $(a)$, the parent of $u_0$ is not in $M_0$
Therefore $A_1$ only contains children of $u_0$.
Condition $(b)$ implies that every descendant of $u_0$ is in $a(M_0)$.
Therefore $A_{1} \subseteq a(M_1)$ and each vertex in $A_{1}$ is adjacent to $u_0$.
Thus by the contrapositive of Observation \ref{obs:a2ImpliesIso}, each vertex in $A_{1}$ is in $a_1(M_0)$.
Thus $A_{1}$ is a non-empty set containing all vertices in $a_1(M_k)$ which neighbour $u_0$.
Now set $N_{1}$ to be all vertices in $N_1(M_0)$ which neighbour a vertex in $A_{1}$.
That is let $N_1=N(A_1) \cap N_1(M_0)$.
Note that each vertex in $A_1$ has $u_0$ as a parent.
Thus all vertices in $A_1$ have the same depth in $T$.
Moreover, each vertex in $N_1$ is a child of some vertex in $A_1$.
Now let $M_1=(M_0-A_1) \cup N_1$ and note that all conditions of Lemma \ref{lem:move_a1_to_n1} have been satisfied.
By Lemma \ref{lem:move_a1_to_n1} $(i)$, $|M_1| \geq |M_0|$.
By Lemma \ref{lem:move_a1_to_n1} $(ii)$, every vertex in $A_1$ and each of the descendants of $A_1$ are dominated in $M_1$.
The only other vertex which may not be dominated in $M_1$ is $u_0$.
Thus as $u_0 \in M_1$ then $M_1$ is a dominating set.

We will now show that condition $(a)$ holds for $M_1$.
Let $x \in M_1-a(M_1)$ be a supported vertex in $M_1$.
Recall that $A_1$ is set of all neighbours, thus $u_0$ has no neighbours in $M_1$.
It follows from Observation \ref{obs:IsoImpliesCrit} that $u_0 \in a(M_1)$.
In particular $x \neq u_0$.
If $x$ is not a descendant of $u_0$, then $x$ was previously supported in $M_0$ and unaffected by the addition of $A_1$ and removal of $N_1$.
That is the parent of $x$ is still not in $M_1$.
So suppose $x$ is a descendant of $u_0$. 
By condition $(b)$, no descendants of $u_0$ were supported in $M_0$.
Thus $x$ was not supported in $M_0$ and is now supported in $M_1$.
Therefore either $x \notin M_0$ or $x \in a(M_0)$.
However the only vertices in $M_1$ which were not in $M_0$ are in $N_1$.
By Lemma \ref{lem:move_a1_to_n1} $(iii)$, no vertex in $N_1$ is adjacent to any other vertex in $M_1$.
Therefore by Observation \ref{obs:IsoImpliesCrit} $N_1 \subseteq a(M_1)$ and hence $x \notin N_1$.
Thus $x \in M_0$ and so $x \in a(M_0)$.
More specifically $x \in a(M_0)$ and is now supported in $M_1$.
By Lemma \ref{lem:move_a1_to_n1} $(iv)$, $x$ is a grandchild of some vertex in $N_1$ and does not have its parent in $M_1$.
Therefore no parent of a supported vertex is in $M_1$.

We now show condition $(b)$ holds for $M_1$.
Let $x,y \in M_1-a(M_1)$ be a supported vertices in $M_1$.
In the previous paragraph we showed that each supported vertex in $M_1$ is either not a descendant of $u_0$ or is a grandchild of some vertex in $N$.
If both $x$ and $y$ are not descendants of $u_0$, then they were both supported in $M_0$.
Therefore as condition $(b)$ held for $M_0$, then $x$ and $y$ are not descendants of each other.
If both $x$ and $y$ are grandchildren of vertices in $N_1$ then $x$ and $y$ have the same depth.
Thus $x$ and $y$ are not descendants of each other.
Lastly, without loss of generality, suppose $x$ is a grandchild or some vertex in $N_1$ and $y$ is not a descendant of $u_0$.
Then the only way for $x$ and $y$ to be descendants is if $u_0$ was a descendant of $y$.
However $y$ was supported in $M_0$ and $u_0$ was the supported vertex of least depth in $M_0$.
Therefore we have a contradiction and $x$ and $y$ can not be descendants.

Now, if $M_1$ is a minimal dominating set then the algorithm terminates. 
Otherwise we search for a supported vertex $u_1\in M_1-a(M_1)$ of least depth.
We have established that $u_1 \neq u_0$ and $u_1$ at least at the same depth as $u_0$.
Moreover, and supported vertices created by further iterations of the algorithm will have depth strictly lower than $u_0$.
The above arguments will hold for any further iterations.
Thus as the tree is finite, there will eventually by no supported vertices.
Hence the algorithm will terminate with a minimal dominating set which at least as large as $M_0$. 
\end{proof}

We are now ready to reconfigure some minimal dominating sets! It has been established that for any minimal dominating set $a(M)=M$. Thus if $M$ is a minimal dominating set then at least one of $a_1(M)$ or $a_2(M)$ is non-empty. We will give two operations which can form a larger minimal dominating set. One operation reconfigures vertices from $a_1(M)$ and the other operation reconfigures vertices form $a_2(M)$.

\begin{theorem}
\label{thm:a1_to_n1}
Let $M$ be a minimal dominating set  of a tree $T$ with $v \in a_1(M)$. Then Algorithm \ref{alg:MakeMinimal} will will output minimal dominating set $M_i$, with $|M_i| \geq |M_0|\geq |M|$, for inputs:

\begin{itemize}
\item[•] $T$ rooted at $v$, and
\item[•] $M_0=(M-v) \cup N$,
\end{itemize}

\noindent where $N=N[v] \cap N_1(S)$.
\end{theorem}

\begin{proof}
To show Algorithm \ref{alg:MakeMinimal} terminates, it is sufficient to show that conditions $(a)$ and $(b)$ of Theorem \ref{thm:algterm} are satisfied.
First note that $A$ only contains one vertex $v$, so every vertex in $A$ has the same depth in $T$.
Thus $M_0$ is obtained by applying the operation from Lemma \ref{lem:move_a1_to_n1} on $M$ with $A = \{v\}$ and $N=N[v] \cap N_1(M)$.
Hence by Lemma \ref{lem:move_a1_to_n1}$(i)$, we have $|M_0| \geq |M|$.
By Lemma \ref{lem:move_a1_to_n1}$(ii)$, vertex $v$ and all of it descendants are dominated in $M_0$.
As $T$ is rooted at $v$ then $M'$ is a dominating set.
Moreover, every vertex in $N$ is a child of $v$ so properties $(iii)$ and $(iv)$ of Lemma \ref{lem:move_a1_to_n1} apply.
As $M$ was a minimal dominating set, then every vertex in $M_0$ was in $a(M)$ except for the vertices of $N$.
By Lemma \ref{lem:move_a1_to_n1}$(iii)$, we have that $N \subseteq a(M_0)$.
Therefore each supported vertex in $M_0$ was previously in $a(M)$.
By Lemma \ref{lem:move_a1_to_n1} $(iv)$, any vertex in $a(M)$ which is now supported in $M_0$ is a grandchild of some vertex in $N$ and does not have its parent in $M_0$.
Therefore condition $(a)$ of Theorem \ref{thm:algterm} is satisfied.
Moreover, each vertex in $N$ is a child of some vertex in $A$, thus they are each at the same depth in $T$.
Thus each supported vertex in $M_0$ is grandchild of $N$ and hence all at the same depth in $T$.
Therefore no supported vertex is a descendant of another supported vertex in $M_0$.
So condition $(b)$ of Theorem \ref{thm:algterm} is also satisfied.
Therefore by Theorem \ref{thm:algterm}, Algorithm \ref{alg:MakeMinimal} will terminate with a minimal dominating set $M_i$, where $|M_i| \geq |M_0| \geq |M|$.
\end{proof}

Recall that for any vertex $v \in a_2(S)$, its neighbourhood is contained in $N_2(S)$. The following algorthim will allow us to reconfigure vertices in $a_2(S)$ by considering a subset $X \subseteq N_2(M)$.

\begin{center}
\begin{algorithm}[H]
\label{alg:a2_to_n2_X}
Input a minimal dominating set $M$ of a tree $T$\;
Input a subset $X \subseteq N_2(M)$ \;
Set $A=N(X) \cap a(M)$ \;
Set $A_2 = N(X) \cap a_2(M)$  \;
Set $N=N(A) \cap N_1(M)$  \;
Set $M'=(M\cup X-A) \cup N$  \;
  \For{ $x \in X$}
  {
   Root $T$ at $x$ \;
   Set $T_x$ to be the subtree of $T$ obtained by removing the $A_2$ and their respective descendants\;
   Set $M_x'=M' \cap V(T_x)$ \;
   Set $M_x''$ to the output of Algorithm \ref{alg:MakeMinimal} with inputs $T_x$ rooted at $x$ and $M_x'$  \;
   Set $M'=(M'-M_x') \cup M_x''$  \;
   
  }
Return $M'$  \;
 \caption{An algorithm for a subset $X \subseteq N_2(M)$}
\end{algorithm}
\end{center}

%%%%%GIVE EXAMPLE%%%%%%%%%%
For an example of Algorithm \ref{alg:a2_to_n2_X}, consider $M$ in Figure \ref{fig:Alg2Example} $(a)$.
In this case we let $X=\{x_1, x_2, x_3\}$ so $A=\{a_{11}, a_{12}, a_{13}, a_{21}, a_{22}\}$ and $A_2=\{a_{21}, a_{22}\}$. 
Moreover $N=N(A) \cap N_1(M)$  so $N=\{n_1, n_2, n_3\}$. 
In Figure \ref{fig:Alg2Example} $(b)$, we let $M'=(M\cup X-A) \cup N$. 
For each $x \in X$ we have labelled $T_x$.
Note that only $T_{x_1}$ does not induce a minimal dominating set.
Therefore we apply Algorithm \ref{alg:MakeMinimal} to $T_{x_1}$.
$T_{x_1}$ has exactly one supported vertex $u_0$ which is label in Figure \ref{fig:Alg2Example} $(c)$.
The vertex $u_0$ has exactly one neighbour in $a_1(M_{x_1}')$.
That neighbour is removed and its lone neighbour in $N_1(M_{x_1}')$ is added to form $M_{x_1}''$.
Algorithm \ref{alg:a2_to_n2_X} terminates with the minimal dominating set which appears in Figure \ref{fig:Alg2Example} $(c)$.

\begin{figure}[!h]
\def\c{1}
\centering
\subfigure[$M$]{
\scalebox{\c}{
\begin{tikzpicture}
\begin{scope}[every node/.style={circle,thick,draw}]
    \node[shape=circle,draw=black,fill=white, label=above:$x_1$] (1) at (0,0) {};
    \node[shape=circle,draw=black,fill=black, label=above:$a_{21}$] (2) at (0.75,0.25) {};
    \node[shape=circle,draw=black,fill=white, label=above:$x_2$] (3) at (1.5,0) {};
    \node[shape=circle,draw=black,fill=black, label=above:$a_{22}$] (4) at (2.25,0.25) {};
    \node[shape=circle,draw=black,fill=white, label=above:$x_3$] (5) at (3,0) {};

    %T_x3
    \node[shape=circle,draw=black,fill=black, label=right:$a_{13}$] (6) at (3,-0.5) {};
    \node[shape=circle,draw=black,fill=white, label=right:$n_{13}$] (7) at (3.25,-1) {};
    %\node[shape=circle,draw=black,fill=black] (8) at (3.5,-3.5) {};
    \node[shape=circle,draw=black,fill=white] (9) at (2.75,-1) {};
    \node[shape=circle,draw=black,fill=black] (10) at (2.75,-1.5) {};

    %T_x2
    \node[shape=circle,draw=black,fill=white] (11) at (1.75,-0.5) {};
    \node[shape=circle,draw=black,fill=black] (12) at (1.75,-1) {};
    \node[shape=circle,draw=black,fill=white] (13) at (1.75,-1.5) {};
    
    \node[shape=circle,draw=black,fill=black, label=left:$a_{12}$] (14) at (1.25,-0.5) {};
    \node[shape=circle,draw=black,fill=white, label=left:$n_{12}$] (15) at (1.25,-1) {};
    %\node[shape=circle,draw=black,fill=black,label={[label distance=0.75mm]45:$u_1$}] (16) at (1,-2.5) {};
    
    %T_x1
    \node[shape=circle,draw=black,fill=black, label=left:$a_{11}$] (17) at (0,-0.5) {};
    \node[shape=circle,draw=black,fill=white, label=left:$n_{11}$] (18) at (0,-1) {};
    \node[shape=circle,draw=black,fill=white] (19) at (0,-1.5) {};
    \node[shape=circle,draw=black,fill=black] (20) at (0,-2) {};
    \node[shape=circle,draw=black,fill=black] (21) at (0,-2.5) {};
    \node[shape=circle,draw=black,fill=white] (22) at (0,-3) {};
    
    \node[shape=circle,draw=black,fill=white] (23) at (2.25,-0.5) {};
    \node[shape=circle,draw=black,fill=black] (24) at (2.25,-1) {};
    \node[shape=circle,draw=black,fill=white] (25) at (2.25,-1.5) {};

\end{scope}

\begin{scope}
    \path [-] (1) edge node {} (2);
    \path [-] (2) edge node {} (3);
    \path [-] (3) edge node {} (4);
    \path [-] (4) edge node {} (5);
    \path [-] (5) edge node {} (6);
    \path [-] (6) edge node {} (7);
    %\path [-] (7) edge node {} (8);
    
    \path [-] (6) edge node {} (9);
    \path [-] (9) edge node {} (10);

    \path [-] (3) edge node {} (11);
    \path [-] (11) edge node {} (12);
    \path [-] (12) edge node {} (13);
    
    \path [-] (3) edge node {} (14);
    \path [-] (14) edge node {} (15);
    %\path [-] (15) edge node {} (16);
    
    \path [-] (1) edge node {} (17);
    \path [-] (17) edge node {} (18);
    \path [-] (18) edge node {} (19);
    \path [-] (19) edge node {} (20);
    \path [-] (20) edge node {} (21);
    \path [-] (21) edge node {} (22);
    
    \path [-] (4) edge node {} (23);
    \path [-] (23) edge node {} (24);
    \path [-] (24) edge node {} (25);
\end{scope}
\end{tikzpicture}}}
\qquad
\subfigure[$M'$]{
\scalebox{\c}{
\begin{tikzpicture}

\draw (0,-1.2) ellipse (0.6cm and 2.25cm);
\draw (1.5,-0.6) ellipse (0.55cm and 1.7cm);
\draw (3,-0.6) ellipse (0.55cm and 1.7cm);

\begin{scope}[every node/.style={circle,thick,draw}]
    \node[shape=circle,draw=black,fill=black, label=above:$T_{x_1}$] (1) at (0,0) {};
    \node[shape=circle,draw=black,fill=white] (2) at (0.75,0.25) {};
    \node[shape=circle,draw=black,fill=black, label=above:$T_{x_2}$] (3) at (1.5,0) {};
    \node[shape=circle,draw=black,fill=white] (4) at (2.25,0.25) {};
    \node[shape=circle,draw=black,fill=black, label=above:$T_{x_3}$] (5) at (3,0) {};
    
    %T_x3
    \node[shape=circle,draw=black,fill=white] (6) at (3,-0.5) {};
    \node[shape=circle,draw=black,fill=black] (7) at (3.25,-1) {};
    %\node[shape=circle,draw=black,fill=black] (8) at (3.5,-3.5) {};
    \node[shape=circle,draw=black,fill=white] (9) at (2.75,-1) {};
    \node[shape=circle,draw=black,fill=black] (10) at (2.75,-1.5) {};

    %T_x2
    \node[shape=circle,draw=black,fill=white] (11) at (1.75,-0.5) {};
    \node[shape=circle,draw=black,fill=black] (12) at (1.75,-1) {};
    \node[shape=circle,draw=black,fill=white] (13) at (1.75,-1.5) {};
    
    \node[shape=circle,draw=black,fill=white] (14) at (1.25,-0.5) {};
    \node[shape=circle,draw=black,fill=black] (15) at (1.25,-1) {};
    %\node[shape=circle,draw=black,fill=black,label={[label distance=0.75mm]45:$u_1$}] (16) at (1,-2.5) {};
    
    %T_x1
    \node[shape=circle,draw=black,fill=white] (17) at (0,-0.5) {};
    \node[shape=circle,draw=black,fill=black] (18) at (0,-1) {};
    \node[shape=circle,draw=black,fill=white] (19) at (0,-1.5) {};
    \node[shape=circle,draw=black,fill=black] (20) at (0,-2) {};
    \node[shape=circle,draw=black,fill=black] (21) at (0,-2.5) {};
    \node[shape=circle,draw=black,fill=white] (22) at (0,-3) {};
    
    \node[shape=circle,draw=black,fill=white] (23) at (2.25,-0.5) {};
    \node[shape=circle,draw=black,fill=black] (24) at (2.25,-1) {};
    \node[shape=circle,draw=black,fill=white] (25) at (2.25,-1.5) {};
\end{scope}

\begin{scope}
    \path [-] (1) edge node {} (2);
    \path [-] (2) edge node {} (3);
    \path [-] (3) edge node {} (4);
    \path [-] (4) edge node {} (5);
    \path [-] (5) edge node {} (6);
    \path [-] (6) edge node {} (7);
    %\path [-] (7) edge node {} (8);
    
    \path [-] (6) edge node {} (9);
    \path [-] (9) edge node {} (10);

    \path [-] (3) edge node {} (11);
    \path [-] (11) edge node {} (12);
    \path [-] (12) edge node {} (13);
    
    \path [-] (3) edge node {} (14);
    \path [-] (14) edge node {} (15);
    %\path [-] (15) edge node {} (16);
    
    \path [-] (1) edge node {} (17);
    \path [-] (17) edge node {} (18);
    \path [-] (18) edge node {} (19);
    \path [-] (19) edge node {} (20);
    \path [-] (20) edge node {} (21);
    \path [-] (21) edge node {} (22);
    
    \path [-] (4) edge node {} (23);
    \path [-] (23) edge node {} (24);
    \path [-] (24) edge node {} (25);
\end{scope}
\end{tikzpicture}}}
\qquad
\subfigure[$(M'-M_{x_1}') \cup M_{x_1}''$]{
\scalebox{\c}{
\begin{tikzpicture}
\begin{scope}[every node/.style={circle,thick,draw}]
    \node[shape=circle,draw=black,fill=black] (1) at (0,0) {};
    \node[shape=circle,draw=black,fill=white] (2) at (0.75,0.25) {};
    \node[shape=circle,draw=black,fill=black] (3) at (1.5,0) {};
    \node[shape=circle,draw=black,fill=white] (4) at (2.25,0.25) {};
    \node[shape=circle,draw=black,fill=black] (5) at (3,0) {};
    
    %T_x3
    \node[shape=circle,draw=black,fill=white] (6) at (3,-0.5) {};
    \node[shape=circle,draw=black,fill=black] (7) at (3.25,-1) {};
    %\node[shape=circle,draw=black,fill=black] (8) at (3.5,-3.5) {};
    \node[shape=circle,draw=black,fill=white] (9) at (2.75,-1) {};
    \node[shape=circle,draw=black,fill=black] (10) at (2.75,-1.5) {};

    %T_x2
    \node[shape=circle,draw=black,fill=white] (11) at (1.75,-0.5) {};
    \node[shape=circle,draw=black,fill=black] (12) at (1.75,-1) {};
    \node[shape=circle,draw=black,fill=white] (13) at (1.75,-1.5) {};
    
    \node[shape=circle,draw=black,fill=white] (14) at (1.25,-0.5) {};
    \node[shape=circle,draw=black,fill=black] (15) at (1.25,-1) {};
    %\node[shape=circle,draw=black,fill=black,label={[label distance=0.75mm]45:$u_1$}] (16) at (1,-2.5) {};
    
    %T_x1
    \node[shape=circle,draw=black,fill=white] (17) at (0,-0.5) {};
    \node[shape=circle,draw=black,fill=black] (18) at (0,-1) {};
    \node[shape=circle,draw=black,fill=white] (19) at (0,-1.5) {};
    \node[shape=circle,draw=black,fill=black, label=left:$u_0$] (20) at (0,-2) {};
    \node[shape=circle,draw=black,fill=white] (21) at (0,-2.5) {};
    \node[shape=circle,draw=black,fill=black] (22) at (0,-3) {};
    
    \node[shape=circle,draw=black,fill=white] (23) at (2.25,-0.5) {};
    \node[shape=circle,draw=black,fill=black] (24) at (2.25,-1) {};
    \node[shape=circle,draw=black,fill=white] (25) at (2.25,-1.5) {};
\end{scope}

\begin{scope}
    \path [-] (1) edge node {} (2);
    \path [-] (2) edge node {} (3);
    \path [-] (3) edge node {} (4);
    \path [-] (4) edge node {} (5);
    \path [-] (5) edge node {} (6);
    \path [-] (6) edge node {} (7);
    %\path [-] (7) edge node {} (8);
    
    \path [-] (6) edge node {} (9);
    \path [-] (9) edge node {} (10);

    \path [-] (3) edge node {} (11);
    \path [-] (11) edge node {} (12);
    \path [-] (12) edge node {} (13);
    
    \path [-] (3) edge node {} (14);
    \path [-] (14) edge node {} (15);
    %\path [-] (15) edge node {} (16);
    
    \path [-] (1) edge node {} (17);
    \path [-] (17) edge node {} (18);
    \path [-] (18) edge node {} (19);
    \path [-] (19) edge node {} (20);
    \path [-] (20) edge node {} (21);
    \path [-] (21) edge node {} (22);
    
    \path [-] (4) edge node {} (23);
    \path [-] (23) edge node {} (24);
    \path [-] (24) edge node {} (25);
\end{scope}
\end{tikzpicture}}}

\caption{Example of Algorithm \ref{alg:a2_to_n2_X}}%
\label{fig:Alg2Example}%
\end{figure}

\noindent For Algorithm \ref{alg:a2_to_n2_X} to return a minimal dominating set, we need additional conditions on $X$. These conditions must guarantee that each vertex in $X$ is critical in $M'$ and each $T_x$ is disjoint. In the next theorem, we provide some sufficient conditions.

\begin{theorem}
\label{thm:a2_to_n2_X}
Let $M$ be a minimal dominating set of a tree $T$ with $X \subseteq N_2(M)$ and $A_2=N(X) \cap a_2(M)$. If
\begin{itemize}
\item[•] $A_2 \cup X$ induces a connected subtree, and
\item[•] $X$ forms an independent set in $T$,
\end{itemize}  

\noindent then Algorithm \ref{alg:a2_to_n2_X} terminates with a minimial dominating set $M'$ with $|M'| \geq |M|-|A_2|+|X|$.
\end{theorem}

\begin{proof}
As $A \cup X$ induces a connected subtree, then each $T_x$ are disjoint otherwise a forming a cycle in $T$.
Note that Algorithm \ref{alg:a2_to_n2_X} will terminates if and only if Algorithm \ref{alg:MakeMinimal} terminates for each input $M_x$ and $T_x$ rooted at $x$.
So fix an $x \in X$, root $T_x$ root at $x$ and let $M_x' = M' \cap V(T_x)$.
We will now show that conditions $(a)$ and $(b)$ of Theorem \ref{thm:algterm} hold for $M_x'$ in $T_x$.
We note that $x$ has no neighbours in $M'$ and therefore no neighbours in $M_x'$.
Thus $x \in a(M_x')$.
Let $C$ be the set of children of $x$ in $T_x$.
Thus every supported vertex in $M_x'$ is either a child of $C \cap M$ or $C - M$.
Suppose $s$ is a supported vertex which is a descendent from $C - M$ in $T_x$.
Then it must be a grandchild of $x$ (i.e. depth 2).
Moreover its parent is in $C-M$ and hence not in $M_x'$.
Suppose $s$ is a supported vertex which is a descendent from $C \cap M$ in $T_x$.
Let $M_x = M \cap V(T_x)$.
Every vertex other than $x$ has the same closed neighbourhood in $T_x$ as it did in $T$.
Thus $M_x$ is a minimal dominating set of $T_x$.
By the definition of $T_x$ every vertex in $C \cap M$ is in $a_1(M_x)$.
As $x \in a_2(M)$ then each vertex in $C \cap M$ still has a neighbour in $N_1(M_x)$.
Thus adding $x$ to $M_x$ does not make any of its children supported.
Thus $C \cap M \subseteq a_1(M_x \cup \{x\})$.
Note that

$$M_x'=(S-A) \cup N$$

\noindent where $S=M_x \cup \{x\}$, $A=C \cap M \subseteq a_1(S)$ and $N=N_{T_x}(A) \cap N_1(S)$.
Moreover, $S$ is a dominating set, each vertex in $A$ has the same depth in $T_x$ and each vertex in $N$ is a child of a vertex in $A$.
Therefore all four parts of Lemma \ref{lem:move_a1_to_n1} apply.
There from $(i)$ we have $|M_x'| \geq |M_x \cup \{x\}|=|M_x|+1$.
From $(iv)$ any supported vertex in $M_x'$ is a grandchild of some vertex in $N$ and does not have its parent in $M_x'$ (i.e. depth 4 in $T_x$).
Therefore every supported vertex in $M_x'$ does not have its parent in $M_x'$.
Hence $(a)$ of Theorem \ref{thm:algterm} holds.
Additionally if it supported in $M_x'$ it is either at depth 2 and a descendant of a vertex in $C-M$ or at depth 4 and a descendant of a vertex in $C \cap M$.
Therefore no supported vertex in $M_x'$ is a descendant of another supported vertex $M_x'$.
Hence $(b)$ of Theorem \ref{thm:algterm} holds.
Therefore Algorithm \ref{alg:MakeMinimal} will return a minimal dominating set $M_x'$ of $T_x$ for each $x \in X$ with $|M_x'| \geq |M_x \cup \{x\}|=|M_x|+1$.

We will now show that Algorithm \ref{alg:a2_to_n2_X} outputs a minimal dominating set $M'$.
We begin by showing $M'$ is a dominating set.
Note that $M' = \bigcup_{x \in X} M_x'$.
Each $M_x'$ is a minimal dominating set of $T_x$.
Thus each vertex which appears in some $T_x$ is dominated.
If a vertex was not in any $T_x$, then it is dominated unless it is also in $N_T[A_2]$.
Each vertex in $A_2$ necessarily has a neighbour in $X$ and is hence dominated.
If a vertex $u \notin X$ is adjacent to a vertex in $A_2$ then it was in $N_2(M)$.
Therefore $u$ had at least two neighbours in $M$.
Thus $u$ is dominating in $M'$ unless every neighbour of $u$ is in $A_2$.
However $A_2 \cup X$ induces a connected tree and hence $u$ having two neighbours in $A_2 \cup X$ woudl form a cycle in $T$.
Therefore $u$ is dominated in $M'$ and $M'$ is dominating set in $T$.
To show $M'$ is minimal, note that the only vertices which are potentially no longer critical in $M_x'$ are in $X$.
Note that the only neighbours of $x$ in $M'$ are in $M_x$ or $X$.
Recall that $x$ has neighbours in $M_x$.
Additionally, $X$ is an independent set.
Hence $x$ has no neighbours in $M'$.
So by Observation \ref{obs:IsoImpliesCrit}, each $x$ is critical in $M'$.

Lastly we show that the final output $M'$ is such that $|M'| \geq |M|-|A_2|+|X|$.
Recall that for each $x \in X$ we showed $|M_x'| \geq |M_x \cup \{x\}|=|M_x|+1$.
Note that each $M_x$ are disjoint and $\bigcup_{x \in X} M_x = M-A_2$.
Therefore when summing over all $x \in X$ we obtain $|M'| \geq |M|-|A_2|+|X|$.

\end{proof}

Theorem \ref{thm:a1_to_n1} and Theorem \ref{thm:a2_to_n2_X} give immediate result regards minimal dominating sets $M$ of the largest size $\Gamma(T)$. This result will allow us to show a matching from $M$ to $V-M$ which saturates $V-M$.

\begin{theorem}
\label{thm:minimal_properties}
If $M$ is a minimal dominating set of a finite tree with $|M|=\Gamma(T)$ then
\begin{itemize}
\item[$(i)$] $|a_1(M)|=|N_1(M)|$
%\item[$(ii)$] If $v\in N_2(M)$ then $N(v) \cap a_2(M) \neq \emptyset$.
\item[$(ii)$] If $X \subseteq N_2(M)$ then $|X| \leq |N(X) \cap a_2(M)|$.
\end{itemize}
\end{theorem}

\begin{proof}
Suppose $M$ is a minimal dominating set with $|M|=\Gamma(T)$.
For each of $(i)$ and $(ii)$, we will use show that if the statement is not true then we can reconfigure $M$ to be a minimal dominating which is larger than $M$.

$(i)$ Suppose not, that is, suppose $|a_1(M)| \neq |N_1(M)|$
Then by Lemma \ref{lem:a1n1} we have $|a_1(M)| < |N_1(M)|$.
Thus by the pigeonhole principle there must be a vertex $v \in a_1(M)$ with at least two neighbours in $N_1(M)$.
By Theorem \ref{thm:a1_to_n1}, Algorithm \ref{alg:MakeMinimal} will produce a minimal dominating set $M_i$ with $|M_i| \geq |M_0|$ where $M_0=(M-v) \cup N$ and $N = N(v) \cap N_1(M)$.
In this case $N \geq 2$, so $|M_0| > |M|$ and hence $|M_i| > |M|$.
However, this contradicts the fact that $|M| = \Gamma(T)$.

$(ii)$ Suppose for some $X \subseteq N_2(M)$ that $|X| > |N(X) \cap a_2(M)|$.
Without loss of generality let $X$ is the smallest such subset of $N_2(M)$.
For simplicity, let $A_2(X)$ denote the set $N(X) \cap a_2(M)$.
Let $T'$ denote the subgraph of $T$ induced by the vertices $X$ and $A_2(X)$.
Note that $T'$ is a connected subtree of $T$, otherwise there would exist a smaller $X' \subset X$ with $|X'| > |A_2(X')|$.
Moreover $X$ must be an independent set in $T$, otherwise we could take a smaller $X'$ which did not induce an edge with $|X'| > |A_2(X')|$.
Therefore by Theorem \ref{thm:a2_to_n2_X}, Algorithm \ref{alg:a2_to_n2_X} terminates with minimal dominating set $M'$ with $|M'| \geq |M|-|A_2(X)|+|X|$. 
As $|X| > |A_2(X)|$ then $|M'| > |M|$.
However, this contradicts the fact that $|M| = \Gamma(T)$.
\end{proof}

The properties exhibited in Theorem \ref{thm:minimal_properties} satisfy Hall's marriage condition.
That is if a maximal dominating set $M$ of a tree $T$ has $|M| = \Gamma(T)$, then there is a matching from $M$ to $V-M$ which saturates $V-M$.
More specifically there are two matchings via the edges of $T$)..
One from $N_1(M)$ to $a_1(M)$ which saturates $N_1(M)$ and another from $N_2(M)$ to $a_2(M)$ which saturates $N_2(M)$.
We will refer to such a matching as a \emph{saturated critical matching}.
The significance of a maximum critical matching is that the addition of any subset $T \subseteq V-M$ to $M$ guarantees that we have reduced the number of critical vertices by at least $T$.
However, if $|M| < \Gamma(T)$ we can not guarantee a saturated critical matching.
Let $\rho_1(M)$ and $\rho_2(M)$ denote the size of largest critical matching from $N_1(M)$ to $a_1(M)$ and from $N_2(M)$ to $a_2(M)$ respectively.
We will bound the number of ``unmatched" vertices by the gap between $|M|$ and $\Gamma(T)$.

\begin{theorem}
\label{thm:unmatched}
Let $M$ be a minimal dominating set of a tree $T$. Then

$$\um(M) \leq 2(\Gamma(T) - |M|),$$

\noindent where $\um(M)=|N_1(M)| - \rho_1(M)+|N_2(M)|-\rho_2(M)$
\end{theorem}
\begin{proof}
%Let $f(M)=|N_1(M)| - \rho_1(M)+|N_2(M)|-\rho_2(M)$.
%Note that $f(M)$ is the number the unmatched vertices in $V-M$
Thus it suffices to show $\um(M) \leq 2(\Gamma(T) - |M|)$
We will prove the statement through induction on the size of $M$.

First suppose $|M|=\Gamma(T)$.
By Theorem \ref{thm:minimal_properties}$(i)$, we have that $|a_1(M)| = |N_1(M)|$.
By definition, each vertex in $a_1(M)$ has at least one neighbour in $N_1(M)$.
Moreover, each vertex in $N_1(M)$ has exactly one neighbour in $a_1(M)$.
Therefore there is a perfect matching from $a_1(M)$ to $N_1(M)$.
Thus $\rho_1(M)=|N_1(M)|$.
Note that Theorem \ref{thm:minimal_properties}$(ii)$ satisfies Hall's marriage condition.
Therefore there is a matching between $a_2(M)$ and $N_2(M)$ which saturates $N_2(M)$.
Thus $\rho_2(M)=|N_2(M)|$.
Thus $\um(M)=0$ and the inequality holds for $|M|=\Gamma(T)$.

Now let $|M| < \Gamma(T)$ and suppose that the inequality holds for all larger minimal dominating sets of $T$.
If $\rho_1(M) = |N_1(M)|$ and $\rho_2(M)=|N_2(M)|$ then clearly $f(M) \leq 2(\Gamma(T) - |M|)$.
So either $\rho_1(M) < |N_1(M)|$ or $\rho_2(M) < |N_2(M)|$.
Thus we are able to use the same construction from the proof of Theorem \ref{thm:minimal_properties} to obtain a strictly larger minimal dominating set $M'$.
By our inductive hypothesis $\um(M') \leq 2(\Gamma(T) - |M'|)$.
Thus it suffices to show that  $\um(M)-\um(M') \leq 2(|M'|-|M|)$.
Fix a maximum critical matching of $M$.
In the worst case, each critical matching was preserved in $M'$, each addition vertex in $M'$ was unmatched in $M$, each additional vertex in $M'$ also matched to a unmatched in $M$.
In this case each addition vertex in $M'$ would corresponds to two less unmatched vertices in $M'$.
Therefore $\um(M)-\um(M') \leq 2(|M'|-|M|)$ and hence our claim holds by induction.
%
%
%
%
%First suppose $\rho_1(M) < |N_1(M)|$.
%Then $v \in a_1(M)$ with at least two neighbours in $N_1(M)$.
%By Theorem \ref{thm:a1_to_n1}, Algorithm \ref{alg:MakeMinimal} will produce a minimal dominating set $M_i$ with $|M_i| \geq |M_0|$ where $M_0=(M-v) \cup N$ and $N = N(v) \cap N_1(M)$.
%Note that in $M$, $v$ can only be matched with one vertex in $N$, and the other vertices in $N$ are unmatched.
%Let $u$ be the vertex in $N$ which is matched with $v$ in a maximum matching.
%After implementing Algorithm \ref{alg:MakeMinimal}, we obtain $v \in N_2(M_i)$ and $N \subseteq a_2(M_i)$.
%Thus the matching from $v$ to $u$ is preserved in $M_i$.
%Moreover, the other previously unmatched vertices in $N$ are now in $M_i$ and hence no longer unmatched.
%Iterating this argument over each step of Algorithm \ref{alg:MakeMinimal} gives that the number of unmatched vertices in $M$ which are now in $M_i$ is at least $|M_i|-|M|$.
%In the most extreme case, each of the previously unmatched vertices of $M$ creates a new matching with another unmatched vertex of $M$ which is still in $V-M_i$.
%Thus $f(M)-2(|M_i|-|M|) \leq f(M_i)$.

\end{proof}

We are now ready for the upper bound on $a(S)$.

\begin{lemma}
\label{lem_a(S)upperbound}
If T is a tree then for any dominating set $S \in \mathcal{D}(T)$ then $|a(S)| \leq 2\Gamma(T)-|S|$.
\end{lemma}

\begin{proof}
For a dominating set $S$ of a tree $T$, let $M$ be the largest minimal dominating contained in $S$.
Fix a maximum critical matching from $V-M$ to $M$.

First suppose $|M|=\Gamma(T)$.
From \ref{thm:unmatched} we have $\um(M)=0$, therefore the matching saturates $V-M$. 
For any vertex $v \in S-M$ let $m(v)$ denote its matched neighbour in $M$.
Note that if $v \in N_1(M)$ then its match neighbour in $m(v) \in a_1(M)$ only has $v$ as a neighbour in $N_1(M)$.
Thus $m(v)$ would become supported if $v$ was added to $M$.
Additionally, if $v \in N_2(M)$ then its match neighbour in $m(v) \in a_2(M)$ would become supported if $v$ was added to $M$.
Thus for each vertex in $S-M$, its matched neighbour would become supported upon its addition to $M$.
Therefore $a(S) \leq |a(M)|-|S-M|$.
As $a(M)=M$ and $|M|= \Gamma(T)$ we get $|a(S)| \leq 2\Gamma(T)-|S|$.

Now suppose $|M| < \Gamma(T)$.
Note that $|a(S)| \leq |M|$.
Therefore $|a(S)| \leq 2\Gamma(T)-|S|$ holds when $|M|\leq 2\Gamma(T)-|S|$.
Equivalently, $|a(S)| \leq 2\Gamma(T)-|S|$ holds when $|S-M|\leq 2(\Gamma(T)-|M|)$.
So suppose $|S-M| > 2(\Gamma(T)-|M|)$ and let $t=|S-M| - 2(\Gamma(T)-|M|)$.
From \ref{thm:unmatched} we have $\um(M) \leq 2(\Gamma(T)-|M|)$. 
Therefore at least $t$ vertices in $S-M$ are critically matched.
It follows from a similar argument used when $|M|=\Gamma(T)$ that $|a(S)| \leq |a(M)|-t$.
Recall that $M = a(M)$ so we obtain

$$|a(S)| \leq |a(M)|-t = |M|-(|S-M| - 2(\Gamma(T)-|M|))=2\Gamma(T)-|S|.$$
\end{proof}

\begin{theorem}
\label{thm:decreasing}
Let $T$ be a tree of order $n$. Then

$$ d_{\left\lceil\frac{n+2\Gamma(T)-2}{3} \right\rceil} \geq \cdots\geq d_{n-1} \geq d_{n},$$

\noindent where $d_i$ denotes the number of dominating sets in $T$ of size $i$.
\end{theorem}

\begin{proof}
For any dominating set $S$ it follows from from Lemma \ref{lem_a(S)upperbound} that  $a(S) \leq 2\Gamma(T) - |S|$.
Therefore $ a(T,i) \leq (2\Gamma(T) - i) d_i$.
By Observation \ref{obs:adi}, $d_{i-1} \geq d_{i}$ if and only if $a(T,i) \leq (2i-n-1)d_i$.
Thus $d_{i-1} \geq d_i$ when $i \geq \frac{n+2\Gamma(T)+1}{3}$.
As $i$ must be an integer, we obtain our result. 
\end{proof}

If follows from Theorem \ref{thm:increasing} and Theorem \ref{thm:decreasing} that the domination polynomial of nearly-well dominated trees is unimodal.

\begin{corollary}
Let $T$ be a tree. If $\Gamma(T) - \gamma(T) < 3$ then $D(T,x)$ is unimodal.
\end{corollary}

\section{Conclusion}
In this paper we showed that not all trees have log-concave domination polynomial.
Additionally we made some progress on the unimodality conjecture in \cite{IntroDomPoly2014} by showing nearly-well-dominated trees have unimodal domination polynomial.
Certainly more investigation into the unimodality conjecture for trees is warranted.

We conclude by extending our results to a related average graph parameter.
The average size of a dominating set in a graph, $\avd(G)$, is typically calculated as one might expect. That is, find all dominating sets in the graph and then take the average of their sizes. That is

$$\avd(G) = \sum\limits_{S \in \mathcal{D}(G)}\frac{|S|}{|\mathcal{D}(G)|}.$$

\noindent  In \cite{2021Beaton}, it was shown that $\avd(G)$ can be calculated using by adding up the total number of critical vertices over all dominating sets. That is

$$\avd(G) = \frac{n}{2} + \sum\limits_{S \in \mathcal{D}(G)}\frac{|a(S)|}{2|\mathcal{D}(G)|}$$

\noindent The bounds for $a(S)$ found in Lemma \ref{lem_a(S)lowerbound} and Lemma \ref{lem_a(S)upperbound} then give us new bounds on $\avd(G)$

\begin{theorem}
\label{thm:avdbound}
Let $T$ be a tree. Then

$$\frac{n+2\gamma(T)}{3} \leq \avd(T) \leq \frac{n+2\Gamma(T)}{3}.$$
\end{theorem}

\bibliographystyle{abbrv}
\bibliography{mybib}

\end{document}